\documentclass[10pt]{amsart}

\usepackage{verbatim}

\usepackage{latexsym}
\usepackage{enumerate}
\usepackage{amsmath,amssymb,amsthm,amsfonts,latexsym}
\usepackage{amsmath,amscd}
\usepackage{pstricks, pst-node, pst-text, pst-3d}
\usepackage[bookmarks]{hyperref}
\usepackage{float}
\hypersetup{backref, colorlinks=true}

\let\=\partial

\newtheorem {prop}{Proposition} [section] %[subsection]
\newtheorem {thm}[prop]{Theorem}% [section]
\newtheorem {cor}[prop]{Corollary}%[section]
\newtheorem{lem}[prop]{Lemma}

\theoremstyle{definition}
 \newtheorem {rk}[prop]{Remark}%[section]
\newtheorem {df}[prop]{Definition}%[section]
\newtheorem {dfs}[prop]{Definitions}
\newtheorem {ex}[prop]{Example}
\newtheorem {notation}[prop]{Notation}
\newtheorem {conv}[prop]{Convention}

\newcommand{\R} {\mathbb{R}}

\newcommand{\N} {\mathbb{N}}
\newcommand{\C} {\mathcal{C}}

\newcommand{\eps}{\varepsilon}

\newcommand{\U}{\mathcal{U}}

\newcommand{\dbar}{\overline d}

\newcommand{\pa}{\partial}
\newcommand{\Vol}{\text{Vol}}



\author{Leonid Shartser}

\address
{Department of Mathematics, University of Toronto, 40 St. George
st, Toronto, ON, Canada M5S 2E4 }

\email{shartl@math.toronto.edu}


\thanks{}


\dedicatory{To the third anniversary of P. Milman's working seminar\\
 for his graduate students and postdocs}

\begin{document}
\title[Poincar\'e Inequality for differential forms on manifolds]{Note on explicit proof of Poincar\'e Inequality for differential forms on manifolds}
\maketitle
\begin{abstract}
We prove a Poincar\'e type inequality for differential forms on compact manifolds
by means of a constructive 'globalization' of a local Poincar\'e inequality on convex sets. 
\end{abstract}

\section{ Introduction }
In a recent paper V. Goldshtein and M. Troyanov [GoTr] proved Sobolev-Poincar\'e type inequality for differential forms on compact Riemannian manifolds. 
In this article we present
a constructive alternative method of proof. The latter allows, in particular,
to estimate the constants in the inequalities in geometric terms.
Namely, we construct for any smooth $r$-form $\omega$ on a Riemannian manifold 
$M$ a smooth $r$-form $\xi$ on $M$ such that $d \omega = d\xi$ and inequality
\begin{equation} \label{eq1}
\| \xi \|_{L^p(M)} \leq C \| d\omega \|_{L^q(M)}
\end{equation}
holds for $p$ and $q$ in a certain (standard) range with a positive
constant $C$ depending only on $p$, $q$, $r$ and manifold the $M$ 
(Theorems \ref{loc_Poincare_ineq} and \ref{global_Poincare}).
The structure of the proof is first to show inequality (\ref{eq1})
locally by means of adapting a proof of Lemma 3.11 from [BoMi] to our
setting with differential forms, and then, globalizing it by means of
a novel method that we present in Section \ref{globalization}. 

We are mainly interested in Poincar\'e type inequalities due to the geometric information
that they encode. Our primary goal is to study such inequalities on singular sets of algebraic 
nature, such as semialgebraic sets, in order to better understand the metric behavior 
of such sets. Constructive proofs of such inequalities 
would, hopefully, allow to extend results of this type to a singular setting.
The results of this article were announced in [S].

\begin{comment}
Our study of Poincar\'e type inequalities was initiated by an observation that
inequality (\ref{eq1}) implies an embedding of $L^{\overline p}$ cohomology
(see Section \ref{Lp_cohomology}) into the standard smooth De Rham cohomology.
In particular, Poincar\'e inequality for a non compact manifold, such as regular 
part of an algebraic set, would imply that its $L^{\overline p}$ cohomology is finite dimensional. 
\end{comment}

\begin{comment}
We are mainly interested in Poincar\'e type inequalities due to the geometric information
that they encode. Our primary goal is to study such inequalities on singular sets of algebraic 
nature, such as semialgebraic sets, in order to better understand the metric behavior 
of such sets. We hope that constructive proofs of such inequalities 
is the first step towards generalizing them to a singular setting.
%In order to study their metric properties near their singularities. 
In a forth coming paper we prove Poincar\'e type inequality for $p>>1$ on a 
semialgebraic set. 
\end{comment}
\medskip
%Parts of the results of this paper will appear in [S].

Throughout this chapter we will use the following notations.
\begin{notation}
${\ }$\\
\begin{itemize}
\item
Suppose that $X$ is a set and $f,g: X\to \R$ are two functions. We write $f\lesssim g$ if
there exists a positive constant $C$ such that $f\leq C g$.
\item
The symbol $\N$ will denote the set of natural numbers $\{1,2,\dots\}$.
\item If $A$ is a measurable subset of $\R^n$ we write $\text{Vol}(A)$ to denote its 
$n$-dimensional volume.
\item
If $p>1$ denote by $p'$  its H\"older conjugate, that is, $1/p+1/p'=1$.
\item
If $x,y\in\R^n$, we write $d(x,y):=|x-y|$.
\end{itemize}
\end{notation}

\vspace*{4mm}
\noindent \textbf{Acknowledgment. }
We would like to thank P. Milman and A. Nabutovsky for helpful discussions.
\vspace*{2mm}

\section{Local Poincar\'e inequality}
In this section we prove a local Poincar\'e type inequality for differential forms.
That is, we prove inequality (\ref{eq1}) with $M$ being a convex set.
This inequality is well known and was studied, e.g. , in [IwLu].
Our proof of local inequality (\ref{eq1}) utilizes a slightly different approach from the one used in [IwLu].
We show that Poincar\'{e} inequality for differential forms (Theorem \ref{loc_Poincare_ineq}) is a simple consequence of 'universal' inequality (Proposition \ref{lem311}) that extends Lemma 3.11 from [BoMi] to differential forms.

%We extend Lemma 3.11 from [BoMi] to the case of differential forms (see Proposition \ref{lem311}) and then, as a corollary, we obtain the local Poincar\'e inequality for differential forms. 

Suppose that $M$ is an orientable Riemannian manifold. 
We denote by $\Omega^\bullet(M)$ the algebra of smooth differentiable forms on $M$.
Define an $L^p$ norm of a form $\omega\in\Omega^r(M)$ by $\| \omega\|_{L^p}:= \left(\int_M |\omega|^p d\text{Vol}\right)^{1/p}$ where $|\omega|$ denotes the pointwise norm of $\omega$ and $d\text{Vol}$ denotes the volume form on $M$.

\subsection{Poincar\'e inequality on a convex set in $\R^n$}
Let $D\subset\R^n$ be a convex set. For each $y\in D$ define a homotopy operator 
$$ K_y: \Omega^r(D)\to\Omega^{r-1}(D) $$
by the following formula.
$$ K_y\omega := \int_0^1 \psi_y^*\omega dt ,$$
where $ \psi_y :D\times [0,1] \to D $, $\psi_y(x,t):=tx+(1-t)y$.
It is easy to check that $dK_y\omega+ K_yd\omega=\omega$.

%In order to prove an analogue of Poincar\'{e} inequality for differential forms, we 
%need the following estimate.  It is a generalization of Lemma 3.11 from [BoMi].

%Poincar\'{e} inequality for differential forms (Theorem \ref{loc_Poincare_ineq}) is a 
%simple consequence of the following 'universal' inequality that extends Lemma 3.11 from [BoMi]
%to differential forms with nearly the same proof, i.e. by interchanging the order of 
%integrations on the left hand side of the inequality.
The next proposition is an extension of Lemma 3.11 from [BoMi] 
to differential forms with nearly the same proof, i.e. by interchanging the order of 
integrations on the left hand side of the inequality.

\begin{prop}\label{lem311}
Let $D$ be a convex compact set in $\R^n$, $r\in \N\cup\{0\}$ and $p,q\geq 1$ such that \\
%\begin{eqnarray*}
$\text{\ \ \ \ \ \ \ \ \ \ \ \ \ \ \ \ \ \ \ \ }(i)\ \ p\geq q \text{ and } \  \frac{1}{q}-\frac{1}{p}<\frac{1}{n}$ \\
or\\
$\text{\ \ \ \ \ \ \ \ \ \ \ \ \ \ \ \ \ \ \ \ }(ii)\ \  p<q $.
%\end{eqnarray*}

Then, 
$$\left\|\frac{1}{\Vol (D)^{1/p}} \left\| \frac{K_y d\omega(x)}{d(x,y)} \right\|_{L^q(dy)} \right\|_{L^p(dx)} \leq C(p,q,r,n) \left\| d \omega\right\|_{L^q(dx)}$$
for any $r$-form $\omega$, where 
$$
%C(p,q,r,n) is a constant depending only on $p,q,r$ and $n$.
C(p,q,r,n):= 
\left\{ \begin{array}{rl}
 \int_0^1 \min (t^{n/p},(1-t)^{n/p})  t^{r-n/p}(1-t)^{-n/q}dt  &\mbox{ in case  $(i)$} \\
 \int_0^1  \min(t^{n/q},(1-t)^{n/q})t^{r-n/q}(1-t)^{-n/q} dt &\mbox{ in case $(ii)$}
       \end{array} \right.
$$

\end{prop}

\begin{proof}
For a $k$-form $\alpha$ the pullback $\psi_y^*\alpha(x,t)$ can be written in the form 
$$ \alpha_0+dt\wedge\alpha_1 .$$
Denote by $|\psi_y^*\alpha(x,t)|_{1}$ the pointwise norm $|\alpha_1(x,t)|$.
\\
Suppose that $p\geq q$.
\begin{eqnarray*}
\left\| \left\| \frac{K_y d\omega(x)}{|x-y|} \right\|_{L^q(dy)} \right\|_{L^p(dx)}
&=&
\left\{ \int_D \left\| \frac{K_y d\omega(x)}{|x-y|} \right\|_{L^q(dy)}^p dx\right\}^{1/p}
\\&=&
\left\{ \int_D \left( 
\int_D \left| \int_0^1\frac{\psi_y^*d\omega(x,t)}{|x-y|}dt\right|^q dy
\right) ^{p/q} dx\right\}^{1/p}
\\&\leq&
\left\{ \int_D \left( 
\int_D \left[ \int_0^1\frac{\left|\psi_y^*d\omega(x,t)\right|_1}{|x-y|}dt \right]^q dy
\right) ^{p/q} dx\right\}^{1/p}
\\&\leq&
\left\{ \int_D \left( 
\int_0^1 \left[ \int_D  \frac{\left|\psi_y^*d\omega(x,t)\right|_1^q}{|x-y|^q}dy \right]^{1/q}dt
\right) ^{p} dx\right\}^{1/p}
\\&\leq&
\int_0^1 \left\{ \int_D 
\left[ \int_D  \frac{\left|\psi_y^*d\omega(x,t)\right|_1^q}{|x-y|^q} dy \right]^{p/q}
dx\right\}^{1/p}dt.
\end{eqnarray*}

Observe that if $d\omega_1(x,t)$ is the component of $\psi_y^*d\omega(x,t)$ that contains $dt$,
then for a collection of vectors $\xi_1,\dots,\xi_k$ we have 
$$ d\omega_1(x,t)(\xi_1,\dots,\xi_k)=t^k d\omega(\psi_y(x,t); x-y,\xi_1,\dots,\xi_k). $$
It follows that 
$$ \left|\psi_y^*d\omega(x,t)\right|_1 \leq t^k|x-y|\ |d\omega(\psi_y(x,t))| .$$
Therefore,
\begin{eqnarray}\label{l_312}
\int_0^1 \left\{ \int_D 
\left[ \int_D  \frac{\left|\psi_y^*d\omega(x,t)\right|_1^q}{|x-y|^q} dy \right]^{p/q}
dx\right\}^{1/p}dt
&\leq&\nonumber \\
\int_0^1 \left\{ \int_D 
\left[ \int_D  t^{kq}\left|d\omega(\psi_y(x,t))\right|^q dy \right]^{p/q}
dx\right\}^{1/p}dt &=&
\end{eqnarray}
Set $u=tx$ and $v=(1-t)y$ obtaining $du=t^n dx$ and $dv=(1-t)^n dy$
\begin{eqnarray*}
\int_0^1 \left\{ \int_{tD} 
\left[ \int_{(1-t)D}  \left|d\omega(u+v)\right|^q dv \right]^{p/q}
du \right\}^{1/p}t^{k-n/p}(1-t)^{-n/q}dt &=&
\end{eqnarray*}
set $z=u+v$,
\begin{eqnarray*}
\int_0^1 \left\{ \int_{tD} 
\left[ \int_{u+(1-t)D}  \left|d\omega(z)\right|^q dz \right]^{p/q}
du \right\}^{1/p}t^{k-n/p}(1-t)^{-n/q}dt.
\end{eqnarray*}
Note that $D$ is convex we have $u+(1-t)D\subset D$. Let us examine the expression in $\{\  \}$.
\begin{eqnarray*}
\left\{ \int_{tD} 
\left[ \int_{u+(1-t)D}  \left|d\omega(z)\right|^q \phi(y(z-u))dz \right]^{p/q}
du \right\}^{1/p} 
&=&\\
\left\{ \int_{tD} 
\left[ \int_{D} {\bf 1}_{u+(1-t)D}(z)  \left|d\omega(z)\right|^q dz \right]^{p/q}
du \right\}^{1/p} 
&=&\\
\left\| 
 \int_{D} {\bf 1}_{u+(1-t)D}(z)  \left|d\omega(z)\right|^q 
dz \right \|_{L^{p/q}(tD,du)} ^{1/q} 
&\leq&\\
\left( \int_{D} \left\| 
  {\bf 1}_{u+(1-t)D}(z)  \left|d\omega(z)\right|^q 
 \right \|_{L^{p/q}(tD,du)}dz  \right)^{1/q} 
&\leq&\\
\left( \int_{D} \left|d\omega(z)\right|^q  \left\| 
  {\bf 1}_{u+(1-t)D}(z)  
 \right \|_{L^{p/q}(tD,du)}dz  \right)^{1/q}. 
\end{eqnarray*}
Now, 
%let $M:=\sup_D \phi^{p/q}$ and 
consider the following estimate.
\begin{eqnarray*}
\left\| 
  {\bf 1}_{u+(1-t)D}(z)  
 \right \|_{L^{p/q}(tD,du)}
&\leq&
\left\| 
 {\bf 1}_{u+(1-t)D}(z) \right \|_{L^{p/q}(tD,du)}
\\&\leq&
\left( \text{Vol}(D) \min (t^n,(1-t)^n) \right)^{q/p}.
\end{eqnarray*}
Summarizing all the computations we finally obtain
\begin{eqnarray*}
&\ &\left\| \left\| \frac{K_y d\omega(x)}{|x-y|} \right\|_{L^q(dy)}\right\|_{L^p(dx)}
\leq
\\&\leq& 
\int_0^1 \left( \int_{D} \left|d\omega(z)\right|^q  
\left( \text{Vol}(D) \min (t^n,(1-t)^n) \right)^{q/p}dz
\right)^{1/q}  t^{k-n/p}(1-t)^{-n/q}dt
\\&\leq&
(\text{Vol}(D))^{1/p}
\int_0^1 
%\left( \int_{D} \left|d\omega(z)\right|^q   dz\right)^{1/q}
\|d\omega\|_{L^q(D)}
\min (t^{n/p},(1-t)^{n/p})  t^{k-n/p}(1-t)^{-n/q}dt 
\\&\leq&
(\text{Vol}(D))^{1/p}C(p,q,k,n)\|d\omega\|_{L^q(D)}.
\end{eqnarray*}

Now suppose that $p<q$. Up to equation (\ref{l_312}) everything is the same.
Let
$$I:= \left\{ \int_D 
\left[ \int_D  \left|d\omega(\psi_y(x,t))\right|^q dy \right]^{p/q}
dx\right\}^{1/p} $$
By H\"{o}lder inequality with exponent $r=q/p$ we have:
\begin{eqnarray*}
I&\leq& 
\left\{ \left( \int_D 1 dx \right)^{1/r'} \left( \int_D 
\left[ \int_D  \left|d\omega(\psi_y(x,t))\right|^q dy \right]^{rp/q}
dx\right)^{1/r}\right\}^{1/p} 
\\&\leq&
\Vol(D)^{\frac{1}{p}-\frac{1}{q}}
\left( \int_D 
\left[ \int_D  \left|d\omega(\psi_y(x,t))\right|^q dy \right]
dx\right)^{1/q} 
=
\end{eqnarray*}
Set $u=tx$, $v=(1-t)y$ and $z=u+v$ obtaining $du=t^n dx$, $dv=(1-t)^n dy$ and $dz=dv$
\begin{eqnarray*}
&=& \Vol(D)^{\frac{1}{p}-\frac{1}{q}}
t^{-n/q}(1-t)^{-n/q}\left( \int_{tD}  \int_{u+(1-t)D}  
\left|d\omega(z)\right|^q dz du \right)^{1/q}
\\&=&
\Vol(D)^{\frac{1}{p}-\frac{1}{q}}
t^{-n/q}(1-t)^{-n/q}\left( \int_{tD}  \int_{D}  
\left|d\omega(z)\right|^q {\bf 1}_{u+(1-t)D}(z) dz du \right)^{1/q}
\\&=&
\Vol(D)^{\frac{1}{p}-\frac{1}{q}}
t^{-n/q}(1-t)^{-n/q}\left( \int_{D}  \left|d\omega(z)\right|^q \int_{tD}  
 {\bf 1}_{u+(1-t)D}(z) du dz \right)^{1/q}
\\&\leq&
\text{Vol}(D)^{1/p}t^{-n/q}(1-t)^{-n/q} \min(t^{n/q},(1-t)^{n/q})\|d\omega\|_{L^q(D)} .
\end{eqnarray*}
The inequality of the lemma follows from here.
\end{proof}

Next, we prove the local Poincar\'e inequality.
\begin{thm}\label{loc_Poincare_ineq}
Suppose that $p,q\geq 1$ are as in Proposition \ref{lem311}. 
Let $\omega$ be a smooth $r$-form on a convex set $D\subset\R^n$. 
There exists an $r$-form $\xi$ on $D$ such that 
$d\omega=d\xi$ and 
$$\left\| \xi \right\|_{L^p(D)} \leq c \left\| d\omega \right\|_{L^q(D)},$$ 
%$$\left\| Ad\omega \right\|_{L^p(X)} \lesssim \left\| d\omega \right\|_{L^q(X)}$$
%for any $k$-form $\omega$, $k\geq 1$, where 
where $c:={\Vol (D)^{1/p-1/q}}C(p,q,k,n) R$, 
with $C(p,q,k,n)$ from Proposition \ref{lem311}.

\end{thm}

\begin{proof}
Define an average homotopy operator $A$ by the formula
$$ A\omega := \frac{1}{\Vol(D)}\int_D K_y\omega dy.$$
Set $\xi=Ad\omega$.
Note that $ dA\omega+Ad\omega=\omega $ and therefore $d\xi=d\omega$. 
Denote by $R$ the diameter of $D$.
Using H\"older inequality and  Proposition \ref{lem311} we obtain the following estimate.
\begin{eqnarray*}
%\left\|\omega(x)-d A\omega (x)\right\|_{L^p(X)} &=& 
\left\| A d \omega \right\|_{L^p(D)} &=&
\left\| \frac{1}{\Vol(D)}\int_D K_y d \omega(x) dy \right\|_{L^p(D,dx)}\\ 
&=&
\frac{1}{\Vol(D)}\left\| \int_D \frac{K_y d \omega(x) }{d(x,y)}d(x,y)dy \right\|_{L^p(D,dx)} \\ 
\text{(H\"older inequality)}&\leq&
\frac{1}{\Vol(D)}
\left\| \left\| \int_D \frac{K_y d \omega(x)}{d(x,y)}\right\|_{L^q(dy)}
\left\|d(x,y)\right\|_{L^{q'}(dy)} \right\|_{L^p(D,dx)} \\
&\leq &
\frac{1}{\Vol(D)}
\sup_{x\in D} \left\|d(x,y)\right\|_{L^{q'}(dy)}
\left\| \left\| \int_D \frac{K_y d \omega(x)}{d(x,y)}\right\|_{L^q(dy)}
\right\|_{L^p(D,dx)} \\
\text{(Proposition \ref{lem311})}&\leq&
{\Vol (D)^{1/p-1/q}}
c(p,q,k,n) R \|d\omega \|_{L^q(D)}.
\end{eqnarray*}
\end{proof}
%
%
%
%
%
%
%
%
%
%
%
%
%
%
%
%
%
%
%
%
%

% <globailzation goes here>
\section{Globalization of Poincar\'e type inequality}\label{globalization}
In this section we describe a constructive method of proof of Poincar\'e type inequality on a compact manifold. The idea of our construction was inspired from the construction of double \v{C}ech-De Rham complex (see [BT]).
  
The main Theorem of this section is
\begin{thm}\label{global_Poincare}(Global Poincar\'e Inequality)
Let $M$ be a compact Riemannian manifold and $\omega$ an exact $r$-form on it. Suppose that 
$p$ and $q$ are as in Proposition \ref{lem311}.
There exists an
$(r-1)$-form $\xi$ on $M$ such that
\begin{equation}\label{prob_1}
 d\xi=\omega\ \  \text{ and }\ \  \| \xi \|_{L^p(M)}\lesssim \| \omega\|_{L^q(M)} .
 \end{equation}
\end{thm}
In what follows we describe the construction of the form $\xi$ from the latter Theorem.
% A complete proof of this Theorem can be found in [S].
We begin with some basic definitions.
\begin{dfs}\label{cech_complex}
Let $M$ be a Riemannian manifold. A subset $D\subset M$ is called {\bf convex} if
for every two points $p,q\in D$ there exists a unique geodesic that connects $p$ with $q$
and lies entirely in $D$.
Let $\U=\{U_i\}$ be a cover of $M$. 
Denote by $U_I$ the set $U_{i_0}\cap\dots\cap U_{i_s}$ where $I=(i_0,\dots, i_s)$.
The cover $\U$ is called {\bf good cover} if
every  $U\in\U$ is convex.
%finite intersection of sets in $\U$ is either empty or convex.
The {\bf nerve complex} of $\U$ is a simplicial complex $(C_j(\U),\pa)$ 
generated by $\{[I] : U_I \neq \emptyset, I=(i_0,\dots,i_{j}) \}$ 
where $\pa : C_{j+1}(\U)\to C_{j}(\U) $ is defined by
$$ \pa [(i_0,\dots,i_{j+1})]:=\sum_{k} (-1)^{k} [({i_0,\dots,\hat{i_k}\dots,i_{j+1}})] .$$
The {\bf \v{C}ech complex } associated with the cover $\U$ is denoted by $(C^j(\U),\delta)$ where $C^j(\U):=Hom (C_j(\U),\R)$
and $\delta:=\pa^* : C^j(\U)\to C^{j+1}(\U)$ is the dual operator to $\pa$. 
%defined by 
%$$(\delta f) [(i_0,\dots,i_{j+1})] = \sum_{k} (-1)^{k} f[({i_0,\dots,\hat{i_k}\dots,i_{j+1}})]. %$$
The $k^{th}$ cohomology group of $C^\bullet(\U)$ is denoted by $H^k(C^{\bullet}(\U))$.
\end{dfs}
\begin{rk}
It is well known that sufficiently small balls in a Riemannian manifold $M$ are convex (see [D] Proposition 4.2). Therefore, there exists a good cover $\U$ for $M$.
\end{rk}

%Let $M$ be a compact manifold and $\omega$
%an exact $r$-form on it. We are looking for an $(r-1)$-form $\xi$ on $M$ such that
%\begin{equation}\label{prob_1}
% d\xi=\omega\ \  \text{ and }\ \  \| \xi \|_{L^p}\lesssim \| \omega\|_{L^q} .
% \end{equation}

%\begin{rk}
%The method developed in this section for globalizing inequality (\ref{prob_1}) goes through 
%for any semialgebraic set $X$ as long 
%as the inequality in (\ref{prob_1}) can be established locally for some class of differential forms on $X$ (see Section \ref{singular_case} for more details).
%\end{rk}
%\medskip

From here on, we will assume that we are in the setting of Theorem \ref{global_Poincare}.
%In this note we describe the construction of the form $\xi$ 
%Let $T$ be a triangulation of $M$. Denote by $T^k$ the set of open simplices of $T$ of dimension $k$.
%The set of vertices of $T$, $T^0$, is identified with the set of numbers $\{1,\dots,N \}$.
%Let $U_i$ be an open star of a vertex $i\in T^0$ (i.e. the union of all open simplices that contain $p_i$ in their boundaries). 
Let $\U:=\{U_i\}$, $i=1,\dots,N$ be a good cover of $M$.
%The sets $U_i$ form an open cover of $M$ with the following properties: 
%\begin{enumerate}[1.]
%\item for each $I=(i_0,\dots, i_s)$ the set $U_I$ is either empty or contractible,
%\item for each $I=(i_0,\dots, i_s)$ the set $U_I$ intersects only one $s$-simplex $(i_0,\dots, i_s)$.
%\end{enumerate}
%
In the definition below we define the \v{C}ech complex associated with the sheaf of smooth $r$-forms on $M$.

\begin{df}
Set 
$$K^{r,0}:=\Omega^r(M),\ \ \ \  K^{r,s} := \bigoplus_{i_0<\dots<i_{s-1}}\Omega^{r}(U_I) $$
Let $\alpha\in K^{r,s}$. Denote  by $\alpha_I$, $I=(i_0,\dots,i_{s-1})$ the components of $\alpha$. Define $\delta: K^{r,s}\to K^{r,s+1}$, $s\geq0$,
$$ (\delta \alpha)_J := \left.\left(\sum_{t=0}^{s} (-1)^{t} \alpha_{j_0\dots \hat{j_t}\dots j_{s}}\right)\right|_{U_J},\ \ \  J=(j_0,\dots,j_{s})\ . $$ 
Define an $L^p$ norm on $K^{r,s}$ as follows.
$$ \|\alpha \|_{L^p(K^{r,s})} := \sum_{i_0<\dots<i_{s-1}} \|\alpha_I \|_{L^p(U_I)}.  $$
\end{df}

\begin{conv}\label{conv_1}
We will use the following convention. If $\alpha\in K^{r,s}$ with components $\alpha_I$, $I=(i_0,\dots,i_{s-1})$, $i_0<\dots<i_{s-1}$ and $\tau$ is a permutation of $\{0,\dots\ s-1\}$ then \mbox{$\alpha_I=\alpha_{\tau(I)}sign(\tau)$}.
%\\
\end{conv}
%\begin{rk}\label{def_delta_0}
%If $\alpha\in \Omega^r(M)$ then $\delta\alpha$ is the element in $K^{r,1}$ with components %$\alpha|_{U_i}$
%\end{rk}
In the next proposition we list fundamental properties of the complex $(K^{r,\bullet},\delta)$.
\begin {prop}\label{delta_exact}
${\ }$
\begin{enumerate}[i.]
\item $(K^{r,\bullet},\delta)$ is a complex, i.e. $\delta^2=0$.
\item $(K^{r,\bullet},\delta)$ is an exact complex, i.e. $\delta$ cohomology of $(K^{r,\bullet},\delta)$ are trivial and moreover, if $\beta\in K^{r,s+1}$ with $\delta \beta =0$ then there exists $\alpha\in K^{r,s}$ such that $\beta=\delta\alpha$ and 
\begin{itemize}
 \item $\| \alpha\|_{L^p(K^{r,s})} \lesssim \|\beta\|_{L^p(K^{r,s+1})} $,
 \item  $\|d\alpha \|_{L^p(K^{r+1,s})} \lesssim \| \beta \|_{L^p(K^{r,s+1})} + \| d\beta \|_{L^p(K^{r+1,s+1})} .$
\end{itemize}
\end{enumerate}
\end{prop}
\begin{proof}
The proof of this proposition without estimates can be found in [BT] Proposition 8.3 and 8.5.
Part $(i)$ follows from a direct computation of $\delta^2$. For part $(ii)$, suppose that $\beta\in K^{r,s+1}$, $\delta\beta=0$. Let $\rho_j$ be a partition of unity subordinate to the cover $\{U_i\}$. Set 
\begin{equation}\label{gluing}
\alpha_{i_0,\dots,i_{s-1}}:= \sum _{j} \rho_j\beta_{j,i_0,\dots,i_{s-1}}. 
\end{equation}
Direct computation shows that $\delta \alpha = \beta$. 
%The estimates of norms of $\alpha$ 
%and $d\alpha$ follow trivially from (\ref{gluing}) (by applying H\"older and triangle inequalities) to expression of norms of $\alpha$ and $d\alpha$.

%\begin{comment}
Now the first estimate can be obtained as follows.
\begin{eqnarray}
\|\alpha \|_{L^p(K^{r,s})} &=& \sum_{i_0<\dots<i_{s-1}}\| \sum _{j} \rho_j\beta_{j,i_0,\dots,i_{s-1}} \|_{L^p(U_I)}     \nonumber \\
&\leq&
\sum_{i_0<\dots<i_{s-1}} \sum _{j} \|  \rho_j\beta_{j,i_0,\dots,i_{s-1}} \|_{L^p(U_I)}\nonumber \\
&\leq &
\sum_{i_0<\dots<i_{s-1}} \sum _{j} \|  \beta_{j,i_0,\dots,i_{s-1}} \|_{L^p(U_{j,I})}\nonumber \\
&=& \| \beta \|_{L^p(K^{r,s+1})}.\nonumber
\end{eqnarray}
And for the second estimate,
\begin{eqnarray}
\|d\alpha \|_{L^p(K^{r+1,s})} &=& \sum_{i_0<\dots<i_{s-1}}\| \sum _{j} d\rho_j\wedge\beta_{j,i_0,\dots,i_{s-1}} + \rho_j\wedge d\beta_{j,i_0,\dots,i_{s-1}}  \|_{L^p(U_I)}     \nonumber \\
&\leq&
\sum_{i_0<\dots<i_{s-1}} \sum _{j}\| d\rho_j\wedge\beta_{j,i_0,\dots,i_{s-1}}\|_{L^p(U_I)}  + \|\rho_j d\beta_{j,i_0,\dots,i_{s-1}}  \|_{L^p(U_I)}    
\nonumber \\
&\lesssim &
\sum_{i_0<\dots<i_{s-1}} \sum _{j}\| \beta_{j,i_0,\dots,i_{s-1}}\|_{L^p(U_{j,I})}  + \|d\beta_{j,i_0,\dots,i_{s-1}}  \|_{L^p(U_{j,I})}    
\nonumber \\
&=&
\| \beta \|_{L^p(K^{r,s+1})} + \| d\beta \|_{L^p(K^{r+1,s+1})}. \nonumber
\end{eqnarray}
%\end{comment}
\end{proof}
%\medskip
%
%
%
%
%
%
Before we give the general construction of the form $\xi$ that satisfies \ref{prob_1} we illustrate 
the construction on an example. 

%Next, we define a collection of elements $\xi^s\in K^{r-s-1,s+1}$ from which we eventually %construct a global solution $\xi$ to (\ref{prob_1}).

\begin{ex}\label{ex1}
Suppose that $\omega$ is a closed
$2$ form on $M$. Consider the following table.
\begin{table}[H]
\caption{Construction of $\xi^k$}
\label{tbl1}
\begin{tabular}{c|c||c|c|c|c|}
\cline{2-6}
             &2   & $\omega \rightarrow $      &$\omega_{i_0}$ & &  \\
\cline{2-6} 
             &1   &            & $\ \ \ \  \begin{array}{l} \uparrow \\ \xi^0_{i_0}\rightarrow  \end{array} $  & $\begin{array}{l} \\ (\delta\xi^0)_{i_0,i_1}\end{array}$ & \\
\cline{2-6} 
$d\ \uparrow$&0   & & & $\begin{array}{l} \uparrow \\ \xi^1_{i_0,i_1}\rightarrow  \end{array} $   & $\begin{array}{l} \\ (\delta\xi^1)_{i_0,i_1,i_2}\end{array}$ \\
%\cline{2-7}
% &0 &  & & & $\begin{array}{l} \uparrow \\ \oplus\xi^2_{i_0,i_1,i_2}\rightarrow  \end{array} $ &   $\begin{array}{l} \\ \oplus(\delta\xi^2)_{i_0,i_1,i_2,i_3}\end{array}$  \\

\cline{2-6}
             &     &$\Omega^\bullet(M)$      & $\bigoplus_{i_0}\Omega^\bullet(U_{i_0})$  &   $\bigoplus_{i_0,i_1}\Omega^\bullet(U_{i_0,i_1}) $&   $\bigoplus_{i_0,i_1,i_2}\Omega^\bullet(U_{i_0,i_1,i_2})$  %$\bigoplus_{i_0,i_1,i_2,i_3}\Omega^\bullet(U_{i_0,i_1,i_2,i_3})$   
             \\
% & &0 &1 &2 &3 &4 \\
 \cline{2-6}
\multicolumn{4}{l}{$\ \ \ \ \ \ \ \ \ \ \ \ \ \delta\ \ \rightarrow$}\\
 %&        &       $\delta\ \ \rightarrow$   & & &\\
\end{tabular}
\end{table}

An entry in the table represents the components of an element in the space indicated in
the the same column at bottom row.
The vertical arrows represent the exterior derivative $d$ and 
the horizontal arrows represent the action of differential $\delta$.
Start off by placing $\omega$ in the first column of the table in the second row (corresponding 
to the degree of the form). 
Apply $\delta$ to $\omega$ to obtain an element $\oplus\omega_{i_0}$ in the second column
of the table. Since $d\omega_{i_0}=0$ and $\U$ is a good cover, we  can apply the 
local Poincar\'e inequality to obtain an element $\xi^0:=\oplus\xi^0_{i_0}$ such that
$$d\xi^0_{i_0}=\omega_{i_0} $$  
and 
$$\|\xi^0_{i_0} \|_{L^p(U_{i_0})}\lesssim \|\omega \|_{L^q(U_{i_0})}\leq \|\omega \|_{L^q(M)} \ \text{ for all } i_0.$$
Next, we apply $\delta$ to $\xi^0$ to get an element $\delta\xi^0:=\oplus(\delta\xi^0)_{i_0,i_1}$.
Observe that 
$$ d\delta\xi^0=\delta d\xi^0 = \delta \delta^0 \omega=0. $$ 
Therefore, once again, we can apply local Poincar\'e inequality to $\delta \xi^0$ to obtain
an element $\xi^1:=\oplus \xi^1_{i_0,i_1}$ such that 
$$d\xi^1 = \delta \xi^0$$  and 
\begin{eqnarray*}
\| \xi^1_{i_0,i_1} \|_{L^p(U_{i_0,i_1})} \lesssim \| (\delta\xi^{0})_{i_0,i_1}\|_{L^p(U_{i_0,i_1})}&\leq& \|\xi^0_{i_0} \|_{L^p(U_{i_0})}+\|\xi^0_{i_1} \|_{L^p(U_{i_1})} \\
&\leq& 2 \|\omega \|_{L^q(M)},
\end{eqnarray*}
for all $i_0,i_1$.
Finally, note that 
$$d \delta\xi^1 = \delta d\xi^1 = \delta \delta\xi^0 = 0  .$$
Since the components $(\delta \xi^1)_{i_0,i_1,i_2}$ of $\delta \xi^1$ are functions with zero exterior derivatives it follows that they are constants.
So far we have only used the fact that $\omega$ is closed.  In order to find
a global form $\xi$  that satisfies (\ref{prob_1}) we have to 
assume that $\omega$ is exact. 
Therefore, in the next step of the construction we assume that $\omega$ is exact and 
find a global $(r-1)$ form $\xi$ that satisfies (\ref{prob_1}).
By Theorem \ref{exist_const} below, there exists an element $c\in\bigoplus \Omega^0(U_{i_0,i_1})$ with constant components $c_{i_0,i_1}$  for all $i_0,i_1$
such that 
$$\delta\xi^1- \delta c =\delta(\xi^1-c)=0.$$ 
%where $(\delta c)_{i_0,i_1,i_2} := c_{i_1,i_2}-c_{i_0,i_2}+c_{i_0,i_1}$.
Moreover, by Corollary \ref{esti_const} we have 
$$ \| c_{i_0,i_1} \|_{L^p(U_{i_0,i_1})}\lesssim \|\omega \|_{L^q (M)}. $$
We will construct (inductively) elements $x^1\in \bigoplus\Omega^0(U_{i_0})$ and 
$x^0\in \Omega^1(M)$ such that $\xi:=x^0$ satisfies (\ref{prob_1}), see Table \ref{tbl2} below.
Note that each row $r$ of the latter table is the complex $(K^{r,\bullet},\delta)$. 
By Proposition \ref{delta_exact} each such row is exact. 
Therefore, by the same proposition, there exists an element $x^1$ such that 
$\delta x^1 = \xi^1-c$ and the following estimates hold
$$ \| x^1\|_{L^p(K^{0,1})}\lesssim\|\xi^1-c \|_{L^p(K^{0,2})} $$
and
$$
\| dx^1\|_{L^p(K^{1,1})}\lesssim\|\xi^1-c \|_{L^p(K^{0,2})}+\|d(\xi^1-c )\|_{L^p(K^{1,2})}.
$$
\begin{table}[H]
\caption{Construction of $x^k$}
\label{tbl2}
\begin{tabular}{c|c||c|c|c|c|}
\cline{2-6}
&2   &$\ \ \omega \rightarrow $ & $\omega_{i_0}$ & &  \\
\cline{2-6} 

&1   & 
$\ \ \ \  \begin{array}{l} \uparrow \\ x^0\rightarrow  \end{array}$ & 
$\ \ \ \  \begin{array}{l}\ \ \ \ \ \uparrow \\ \xi^0_{i_0}-dx^1_{i_0}\rightarrow  \end{array} $& $\begin{array}{l} \\0,\ \ \ \ \ \ \  (\delta\xi^0)_{i_0,i_1}\end{array}$ & \\
\cline{2-6} 

$d\ \uparrow$&0& 
&$\begin{array}{l}  \\ x^1_{i_0}\rightarrow  \end{array} $ & $\begin{array}{l}\ \  \ \ \ \ \ \ \ \  \uparrow \\ \xi^1_{i_0,i_1}-c_{i_0,i_1}\rightarrow  \end{array} $   & $\begin{array}{l} \\ 0\end{array}$ \\
\cline{2-6}
%----------------------------------------------------
&&$\Omega^\bullet(M)$      & $\bigoplus_{i_0}\Omega^\bullet(U_{i_0})$  &   $\bigoplus_{i_0,i_1}\Omega^\bullet(U_{i_0,i_1}) $&   $\bigoplus_{i_0,i_1,i_2}\Omega^\bullet(U_{i_0,i_1,i_2})$\\  
\cline{2-6}
\multicolumn{4}{l}{$\ \ \ \ \ \ \ \ \ \ \ \ \ \delta\ \ \rightarrow$}\\
\end{tabular}
\end{table}
Note that 
$$\delta(\xi^0-dx^1)=d\xi^1-d\delta x^1 = d(\xi^1 - \xi^1 +c) = 0.  $$ 
Hence, by exactness of the second row there exists an element $x^0$ such that
$\delta x^0 = \xi^0-dx^1 $ and we have
$$ \| x^0\|_{L^p(M)}\lesssim\|\xi^0-dx^1 \|_{L^p(K^{1,1})} $$
and
$$
\| dx^0\|_{L^p(M)}\lesssim\|\xi^0-dx^1 \|_{L^p(K^{1,1})}+\|d(\xi^0-dx^1 )\|_{L^p(K^{2,1})}.
$$
Set $\xi:=x^0$. We claim that $d\xi=\omega$. Indeed, 

$$ \delta(\omega- dx^0) = \delta\omega - d\delta x^0 = \delta \omega - d(\xi^0 -dx^1) = 
\delta\omega-d \xi^0 = 0. $$
It follows that $(\omega - dx^0)|_{U_{i_0}} = 0$ for all $i_0$ and therefore $\omega=dx^0$ on $M$.
Moreover, combining all the estimates from above we obtain
\begin{eqnarray*}
\|\xi \|_{L^p(M)}  &\lesssim& \|\xi^0-dx^1 \|_{L^p(K^{1,1})}\\ 
&\lesssim& \|\xi^0\|_{L^p(K^{1,1})} +\|dx^1 \|_{L^p(K^{1,1})} \\
&\lesssim& \|\omega \|_{L^q(M)} + \|\xi^1-c \|_{L^p(K^{0,2})}+\|d\xi^1\|_{L^p(K^{1,2})}\\
&\lesssim& \|\omega \|_{L^q(M)} + \|\xi^1\|_{L^p(K^{0,2})}+\|c \|_{L^p(K^{0,2})}+\|\delta\xi^0\|_{L^p(K^{1,2})}\\
&\lesssim& \|\omega \|_{L^q(M)}.
\end{eqnarray*}
This concludes the example.
\end{ex}
In what follows we give the general construction of 
the forms $\xi^s$ and $x^s$ as in the example above.

\subsection{Construction of elements $\xi^s\in K^{r-s-1,s+1}$}
\begin{df}\label{def_xi}
$ ${}
%\begin{enumerate} [(1)]
%Define $\xi^0$ by setting $\xi_i^0$ to be a solution of (\ref{prob_1}) in $U_i$, $i\in T^0$.
Set $\xi^{-1}:=\omega$ and  define
$\xi^s$ by setting the $I$'th component, $\xi^s_I$, to be a solution of the equation 
\begin{equation}\label{xi_eq_1}
 d\xi^s_I=(\delta\xi^{s-1})_I 
 \end{equation}
  in $U_I$, $I=(i_0,\dots,i_s)$ such that 
\begin{equation} \label{xi_eq_2}
 \| \xi_I^s \|_{L^p(U_I)} \lesssim \| (\delta\xi^{s-1})_I\|_{L^p(U_I)}\ ,
 \end{equation}
for $0\leq s\leq r-1 $.
%\end{enumerate}
\end{df}
We remark that equation (\ref{xi_eq_1}) can be solved with an estimate (\ref{xi_eq_2}) by means of local Poincar\'e inequality since $U_I$ is convex and $d\delta\xi^{s-1}=0$ (cf. Example \ref{ex1}).  
%
%To see that $d\delta\xi^{s-1}=0$ we first let $s\geq 2$ and note that since $d$ commutes with $\delta$ we have
%$$ d\delta \xi^{s-1} = \delta d \xi^{s-1} = \delta d \xi^{s-1} = \delta \delta \xi^{s-2} = 0 .$$
%In the case when $s=1$, $$d \delta \xi^0 = \delta d \xi^0 = \delta \delta^0 \omega = 0. $$
We have the following estimate of $\xi^s$ in terms of the norm of $\omega$:
\begin{prop}\label{estimate_1}
Let $I=(i_0,\dots, i_{s})$. Then,
\begin{equation}
 \| \xi_I^s \|_{L^p(U_I)} \lesssim \|\omega\|_{L^q(M)}.
\end{equation}
\end{prop}
%This proposition can be easily proven by induction on $s$.
%
%
%
%
%\begin{comment}
\begin{proof}
The proof is by induction on $s$. For $s=0$ the statement follows from the local Poincar\'e inequality.
Suppose that $s>0$, we have then
\begin{eqnarray*}
  \| \xi_I^s \|_{L^p(U_I)} &\lesssim& \| (\delta\xi^{s-1})_I\|_{L^p(U_I)} 
  \\&\leq& \nonumber 
 \sum_{t=0}^{s} \| \xi^{s-1}_{i_0\dots\hat{i_t}\dots i_{s-1}}\|_{L^p(U_I)}\\
 &\leq&
 \sum_{t=0}^{s} \| \xi^{s-1}_{i_0\dots\hat{i_t}\dots i_{s-1}}\|_{L^p(U_{i_0\dots\hat{i_t}\dots i_{s}})}
 \\&\leq& \nonumber 
 \sum_{t=0}^{s} \|\omega\|_{L^q(M)}\nonumber\\
  &\lesssim& \|\omega\|_{L^q(M)}. \nonumber  
\end{eqnarray*}
\end{proof}
%\end{comment}
%
%
%
%

Note that $\xi^{r-1}$ is a collection of $0$-forms that satisfy $d\delta\xi^{r-1}=0$.
It means that $(\delta\xi^{r-1})_I$ are constants on each $U_I$, $I=(i_0,\dots,i_{r})$.
(We use the same notation to denote the extension of $(\delta\xi^{r-1})_I$ to a globally defined constant function on $M$.)

\begin{thm}\label{exist_const}
There exists an element $c\in K^{0,r}$ with constant components $c_I$, $I=(i_0,\dots,i_{r-1})$ such that 
\begin{equation}\label{lin_eq}
(\delta c)_I=\sum_{t=0}^{r} (-1)^t c_{i_0,\dots\hat{i_t}\dots,i_r} = (\delta \xi^{r-1})_I,\ \text{ for all } \ I=(i_0,\dots,i_r).
\end{equation}
Moreover, 
there exist $b_{I,L}\in\R$, $I=(i_0,\dots, i_{r-1})$, $L=(l_1,\dots,l_{r})$ such that
$$ c_I = \sum_L b_{I,L} (\delta \xi^{r-1})_L $$
where $b_{I,L}$ depend only on the cover $\U$.

%the constants $C_I$, $I=(i_0,\dots,i_{r-1})$ can be represented as linear combinations of %$(\delta \xi^{r-1})_I$, $I=(i_0,\dots,i_r)$.
\end{thm}
We prove this theorem in subsection \ref{pf_one_exist_const}.
As a consequence of Theorem \ref{exist_const}, we obtain the following corollary.
\begin{cor}\label{esti_const}
The constants $c_I$ from Theorem \ref{exist_const} admit the following estimate 
$$\|c_I\|_{L^p(U_I)} \leq \|\omega\|_{L^q(M)}$$
\end{cor}
\begin{proof}
%There exist constants $b_{I,L}$, $I=(i_0,\dots, i_{r-1})$, $L=(l_1,\dots,l_{r})$ such that
By Theorem \ref{exist_const} we may represent each $c_I$ as $c_I = \sum_L b_{I,L} (\delta \xi^{r-1})_L $. By triangle inequality we have
\begin {eqnarray}\label{r1}
\| c_I\|_{L^p(U_I)}\leq \sum_L |b_{I,L}|\|(\delta\xi^{r-1})_L\|_{L^p(U_I)}
\end{eqnarray}
Observe that $(\delta \xi^{r-1})_L$ is a globally defined constant function and therefore,
similarly to the proof of Proposition \ref{estimate_1} we have
\begin{equation}\label{r2}
\| (\delta \xi^{r-1})_L \|_{L^p(U_I)}= \| (\delta \xi^{r-1})_L \|_{L^p(U_L)}\left( \frac{Vol(U_I)}{Vol(U_L)}\right)^{1/p} \lesssim \| \omega\|_{L^q(M)}.  
\end{equation}
Now, from (\ref{r1}) and (\ref{r2}) we obtain the desired estimate.
\end{proof}

\subsection{Construction of elements $x^s\in K^{r-s-1,s}$}\label{x_construction}{\ }\\
The final step of the construction is to glue all the forms $\xi^s$, $s=0,\dots,r-1$ to a global solution $\xi$ that satisfies (\ref{prob_1}).  We construct inductively forms $x^s\in K^{r-s-1,s}$
such that $\xi:=x^{0}$ is the desired global form (cf. Example \ref{ex1}).
Set $\tilde{\xi}_I^{r-1}=\xi_I^{r-1} - c_I$ where $c_I$ is given by Theorem \ref{exist_const} and $I=(i_0,\dots,i_{r-1})$. Note that $d\tilde{\xi}_I^{r-1}=d \xi_I^{r-1} $ and $\delta \tilde{\xi}^{r-1}=0$. It follows from Proposition \ref{delta_exact} $(ii)$ that there exists 
a form $x^{r-1}\in K^{0,r-1}$ such that $\delta x^{r-1}=\tilde{\xi}^{r-1}$ and 
$$ \| x^{r-1} \|_{L^p(K^{0,r-1})} \lesssim \| \tilde{\xi}^{r-1} \|_{L^p(K^{0,r})}, $$
$$ \| d x^{r-1} \|_{L^p(K^{1,r-1})} \lesssim \| \tilde{\xi}^{r-1} \|_{L^p(K^{0,r})} + 
\| d\tilde{\xi}^{r-1} \|_{L^p(K^{1,r})}.  $$
It follows from Corollary \ref{esti_const} and Proposition \ref{estimate_1} that
\begin{equation} \label{x r-1 esti}
 \| x^{r-1} \|_{L^p(K^{0,r-1})} \lesssim \| \omega \|_{L^q(M)} 
 \end{equation}
and 
\begin{eqnarray}\label{dx r-1 esti}
\| dx^{r-1} \|_{L^p(K^{1,r-1})} &\lesssim& \| \omega \|_{L^q(M)}+\|\delta\xi^{r-2}\|_{L^p(K^{1,r})}
 \\
&\leq& 
\| \omega \|_{L^q(M)} + \|\xi^{r-2}\|_{L^p(K^{1,r-1})}\nonumber \\ 
&\lesssim& \| \omega \|_{L^q(M)} \nonumber.
\end{eqnarray}

Suppose that $x^{r-(t-1)}$ was constructed. By Proposition \ref{delta_exact} $(ii)$ there exists $x^{r-t}$ such that
$$ \delta x ^{r-t} = \xi^{r-t} - dx^{r-t+1}, $$
where 
\begin{equation}\label{estimate_2}
\| x^{r-t} \|_{L^P(K^{t-1,r-t})} \lesssim  \|\xi^{r-t} - dx^{r-t+1} \|_{L^P(K^{t-1,r-t+1})}, 
\end{equation}
and
\begin{eqnarray}\label{estimate_3}
\\
\| dx^{r-t} \|_{L^P(K^{t,r-t})} &\lesssim& 
\|\xi^{r-t} - dx^{r-t+1} \|_{L^P(K^{t-1,r-t+1})} + \|d\xi^{r-t}\|_{L^P(K^{t,r-t+1})}\nonumber \\
&\leq&
\|\xi^{r-t}\|_{L^P(K^{t-1,r-t+1})}+ \|dx^{r-t+1} \|_{L^P(K^{t-1,r-t+1})} + \nonumber\\ & & \|\delta\xi^{r-t-1}\|_{L^P(K^{t,r-t+1})}\nonumber \\
&\lesssim&\nonumber
\| \omega \|_{L^q(M)} + \|dx^{r-t+1} \|_{L^P(K^{t-1,r-t+1})},\nonumber
\end{eqnarray}
provided that $ \delta(\xi^{r-t} - dx^{r-t+1}) = 0 $. Let us verify this condition:
\begin{eqnarray*}
\delta(\xi^{r-t} - dx^{r-t+1}) &=& \delta\xi^{r-t} - d\delta x^{r-t+1}\\ &=&
\delta\xi^{r-t} - d(\xi^{r-t+1}-dx^{r-t+2}) \\&=&
\delta\xi^{r-t} - d\xi^{r-t+1}\\&=&0.\nonumber
\end{eqnarray*}
Using estimates (\ref{x r-1 esti}),(\ref{dx r-1 esti}) (\ref{estimate_2}) and (\ref{estimate_3}) 
we obtain the following proposition.
%one can prove by induction the following proposition.
\begin{prop}\label{x_estimate}
The forms $x^s$ admit the following estimates:
\begin{enumerate}[(1)]
\item $\| x^{r-t}\|_{L^p(K^{t-1,r-t})}\lesssim \| \omega \|_{L^q(M)} . $
\item $\| dx^{r-t}\|_{L^p(K^{t,r-t})}\lesssim \| \omega \|_{L^q(M)} . $
\end{enumerate}
\end{prop}

%\begin{comment}
\begin{proof}
The proof is by induction on $t$. For $t=1$ estimates $(1)$ and $(2)$ are just (\ref{x r-1 esti}) and (\ref{dx r-1 esti}).
Suppose that $t>1$. 
First we prove $(2)$. By (\ref{estimate_3}) and the induction hypothesis we have
\begin{eqnarray}
\\
\| dx^{r-t} \|_{L^P(K^{t,r-t})} 
&\lesssim&\nonumber
\| \omega \|_{L^q(M)} + \|dx^{r-t+1} \|_{L^P(K^{t-1,r-t+1})}\nonumber\\
&\lesssim&
\|\omega\|_{L^q(M)}.\nonumber
\end{eqnarray}
To prove estimate $(1)$ we observe that from Proposition \ref{estimate_1}, (\ref{estimate_2}) and $(2)$ it follows that
\begin{eqnarray}
\| x^{r-t}\|_{L^p(K^{t-1,r-t})} &\lesssim& 
\|\xi^{r-t} - dx^{r-t+1} \|_{L^P(K^{t-1,r-t+1})} \nonumber \\
& \lesssim &
\|\xi^{r-t} \|_{L^P(K^{t-1,r-t+1})} +\| dx^{r-t+1} \|_{L^P(K^{t-1,r-t+1})}\nonumber\\
& \lesssim &
\|\omega\|_{L^q(M)}.\nonumber
\end{eqnarray}
\end{proof}
%\end{comment}
Finally, set $\xi:=x^0$. To see that $dx^0=\omega$ observe that 
$$ \delta(\omega-dx^0)=\delta\omega - d\delta x^0 = \delta\omega - d(\xi^0- dx^1)=
\delta\omega-d\xi^0=0. $$
But $\delta(\omega-dx^0)_i = (\omega-dx^0)|_{U_i}$  and therefore $\omega=dx^0$ on $M$.
The estimate of $\xi$ follows from Proposition \ref{x_estimate} for $t=r$.
\subsection{Proof of Theorem \ref{exist_const}}\label{pf_one_exist_const}$\ $\\
%\begin{conv}
%We extend convention \ref{conv_1} to all objects with indices $C_I$, $I=(i_0,\dots,i_r)$, i.e. %if $\tau$ is a permutation of $\{0,\dots, r\}$ then $C_{\tau(I)}=sign(\tau)C_I$.
%\end{conv}
Note that the linear system of equations (\ref{lin_eq}) has a solution if and only if
$\sum _I (\delta \xi^{r-1})_I a_I =0$  for every $a=\sum a_I [I] \in \ker \pa =\ker \delta^*$.
Therefore, Theorem \ref{exist_const} is equivalent to the following proposition.
\begin{prop}\label{verify_criterion}
Let $a=\sum_I a_I [I] $ be an $r$-cycle then 
$$\sum_I a_I (\delta\xi^{r-1})_I = 0 .$$ 
\end{prop}

In the proof of this proposition we construct an explicit isomorphism from De Rham cohomology to \v{C}ech cohomology of a good cover from which Proposition \ref{verify_criterion} follows immediately.
Another proof of the latter proposition is given in subsection \ref{int_omega}. 
In this subsection we prove 
\begin{prop}\label{int_isom}
The map 
$$ Int : H^r(\Omega^\bullet(M))\to H^r(C^\bullet(\U)) $$
defined by 
\begin{equation}\label{def_Int}
 (Int\ \omega) [I] := (-1)^{\lfloor{\frac{r}{2}\rfloor}}\delta \xi^{r-1}_I,  
 \end{equation}
where $\xi^{r-1}_I$  is defined for the form $\omega$ as in Definition \ref{def_xi} is a well defined isomorphism.
\end{prop}
Before we prove this proposition, we show how Proposition \ref{verify_criterion} follows
from it.
\\
{\it Proof of Proposition \ref{verify_criterion}.}
If $\omega$ is an exact $r$-form then $Int\  \omega= \delta W$. 
If $a=\sum a_I[I]$ is a cycle, then 
$$ (Int\ \omega) a = (-1)^{\lfloor{\frac{r}{2}\rfloor}}\sum a_I \delta\xi^{r-1}_I = (\delta W)a = W \pa a = 0.$$ 
\hfill
$\square$

Next, we prove Proposition \ref{int_isom}.
\begin{proof}
Consider an auxiliary complex $(K^\bullet,D)$ defined by 
$$ K^m :=\bigoplus_{r+s=m} K^{r,s+1} $$
and 
$$ D := d + (-1)^{s}\delta \ \text{ on $K^{r,s+1}$} .$$
This complex is called \v{C}ech-De Rham complex, see [BT] for details.
Denote by $H^r(K^\bullet)$ the $r^{th}$ cohomology group of $K^\bullet$.
The map $\delta :\Omega ^r (M) \to K^r$ induces an  isomorphism 
$$h^0: H^r(\Omega^\bullet(M))\to H^r(K^\bullet)$$
([BT], Proposition 8.8 ).
Similarly, the map $g: C^r (\U) \to K^r $, defined by 
sending an element in $C^r(\U)$ to the corresponding element in $K^{0,r+1}$
induces an  isomorphism 
$$h^1: H^r(C^\bullet(\U))\to H^r(K^\bullet).$$
We claim that $Int = (h^1)^{-1}\circ h^0$. Indeed, let us compute the 
action of $(h^1)^{-1}\circ h^0$ on closed form $\omega$.
First applying $h^0$ to $\omega$ we get an element defined by the $D$ cohomology class of $\delta\omega$.
Note that 
$$ D \xi^{s+1}= d\xi^{s+1} + (-1)^s \delta\xi^{s+1}, $$
%where $\xi^{s+1}$ denotes $\sum_I \xi^{s+1}_I$ and $\delta \xi^{s+1}= \sum_J \delta\xi^{s+1}_J$.
A direct computation, using the latter formula and fact that $\delta \xi^s = d \xi^{s+1}$, shows that
$$ \delta\omega - \sum_{j=0}^k (-1)^{\lfloor{\frac{j}{2}\rfloor}}D\xi^j = (-1)^{\lfloor{\frac{k+1}{2}\rfloor}}\delta\xi^k ,\ \ k\geq 0.$$
It follows from here that $(h^1)^{-1} \delta^0\omega$ is defined by the element
that sends $[I]$ to $(-1)^{\lfloor{\frac{r}{2}\rfloor}}\delta\xi^{r-1}_I $ which is what
was required to prove.
\end{proof}

%
%
%
%
%

%
%
%
%
%
%
%
%
%
%
%
%
%
%
%
%
%
%
%
%\begin{comment}
\subsection{Integral of a closed form over a chain}\label{int_omega}
The purpose of this subsection is to derive a formula that relates an integral of a closed form $\omega$ over a simplex $\sigma$ with integrals of forms $\xi^s$ (defined according to Definition \ref{def_xi}) over simplices in the barycentric subdivision of $\sigma$. 

Let $T$ be a triangulation of $M$. Denote by $\{1,\dots,N\}$ the set of vertices of $T$ 
and set $\U:=\{U_i\}$ where $U_i:=st(i)$, the open star of vertex $i$.
First we will show that Poincar\'e inequality holds near every finite intersection $U_I$.
Then, we will show that the integral of the closed $r$-form $\omega$ over a cycle $a=\sum a_I [I]$  equals to $(-1)^{\lfloor{\frac{r}{2}\rfloor}}\sum a_I (\delta\xi^{r-1})_{I}$, where $[I]$ is identified here with the simplex of $T$ with vertices $(i_0,\dots,i_r)$.
As a corollary, we immediately obtain another proof of Proposition  \ref{verify_criterion}.
Indeed, the form $\omega$ is exact if and only if $\int_a \omega = 0 $ for every $r$-cycle $a$
of $T$. But if $a=\sum a_I [I]$ then 
$$ 0=\int_a\omega= (-1)^{\lfloor{\frac{r}{2}\rfloor}}\sum a_I (\delta\xi^{r-1})_{I}. $$
\begin{prop}
For every $\eps>0$ there exists a neighborhood $U^\eps_I$ of $U_I$ 
such that $d(a,U_I)<\eps$ for every $a\in U^\eps_I$ and if
$\omega$ is a closed $r$-form defined on $U^\eps_I$ then there exists an $(r-1)$ form
$\xi$ on $U^\eps_I$ such that $\omega=d\xi$ and
$$\|\ \xi \|_{L^p(U^\eps_I)}\lesssim\|\omega \|_{L^q(U^\eps_I)} .$$
\end{prop}
\begin{proof}
The set $\overline{U_I}$ can be covered by finitely many balls $B_j$ of radius $\eps/2$.
Let $U^\eps_I$ be the union of those balls. Since the cover $B_j$ is good cover
one can construct elements $\xi^s$ according to Definition \ref{def_xi}.
Denote by $f$ the cochain that sends $[I]$ to $\delta\xi^{r-1}_I$ and note that it 
is a closed. But since $U^\eps_I$ is contractible $f$ is exact. Therefore, 
there exists $(r-1)$ cochain $c$ such that  $f=\delta c$ which means that
$ \delta \xi^{r-1} = \delta c $. Hence, we can construct elements $x^s$ as in subsection \ref{x_construction} such that $\xi:=x^0$ is the desired form.
\end{proof}

The derivation of a formula for an integral of a closed form $\omega$ involves barycentric
subdivision of simplices. In the following definition we introduce our notations for that
purpose.

\begin{dfs}
Let $\sigma=(i_0,\dots,i_r)$ be an $r$-simplex, $J=(j_0,\dots,j_r)$ be a permutation of $\{0,\dots,r\}$ and $t\in\{1,\dots,r\}$.
The {\bf barycenter} of an $s$-face $(i_{j_0},\dots,i_{j_s})$ of $\sigma$ is denoted by
$(i_{j_0},\dots,i_{j_s})^b$.
Set 
$$ Sd_t(J) := \left( (i_{j_0},\dots,i_{j_{t-1}})^b,\dots,(i_{j_0},\dots,i_{j_{r}})^b \right). $$
%The {\bf orientation of $Sd_1(J)$ in $\sigma$} is defined by
%$$ O(Sd_1(J);\sigma):= sign(J) .$$ 
Set $J':=(j_0,\dots,j_{r-1})$. If $K=(k_0,\dots,k_{r-1})$ is a permutation of $\{0,\dots,r-1\}$ then define $J'_K:=(j_{k_0},\dots,j_{k_{r-1}})$.
\end{dfs}
From now to the end of this subsection, assume that $\omega$ is a closed $r$-form on $M$
and $\xi^s$ were constructed according to Definition \ref{def_xi}.
In the next lemma we derive a formula for $\int_\sigma \omega$ where $\omega$ is an closed form.
\begin{lem}\label{formula}
Let $\omega$ be an exact $r$-form, $\xi^s$ be defined as above and $\sigma=(i_0,\dots,i_r)$, then
\begin{eqnarray}
\int_\sigma \omega &=& (-1)^{\lfloor \frac{m+1}{2} \rfloor}\sum_{J:j_0<\dots<j_m}sign(J)
\int_{\pa Sd_{m+1}(J)}\xi^m_{i_{j_0},\dots,i_{j_m}}\nonumber \\ 
&+& %\nonumber \\
(-1)^r\sum_{J:j_0<\dots<j_{r-1}}sign(J)\sum_{t=0}^{m-1}(-1)^{\lfloor \frac{t}{2} \rfloor}
\sum_{K:k_0<\dots<k_t}sign(k)\int_{Sd_{t+1}(J'_K)}\xi^t_{i_{j_{k_0}},\dots,i_{j_{k_t}}},\nonumber
\end{eqnarray}
where $J$ runs over permutations of $\{0,\dots,r\}$ and $K$ runs over permutations of $\{0,\dots,r-1\}$.
\end{lem}
Before we prove this lemma we show how Proposition \ref{verify_criterion} follows from it.
We will need the following combinatorial lemma.
\begin{lem} \label{bdry_rep}
Let $\sigma=(i_0,\dots,i_r)$ be a simplex and $\xi\in K^{t,r}$ then
$$\pa\sigma = \sum_{J:j_0<\dots<j_{r-1}} sign(J)(i_{j_0},\dots,i_{j_{r-1}}),$$
$$ (\delta \xi)_{i_0,\dots,i_r} = \sum_{J:j_0<\dots<j_{r-1}} sign(J)\xi_{i_{j_0},\dots,i_{j_{r-1}}},$$
where $J$ runs over permutations of $\{0,\dots,r\}$.
\end{lem}
\begin{proof}
There exists 1:1 and onto correspondence 
$$ f:\{ s : 0\leq s\leq r \} \to \{ J: J \text{ a permutation of } \{0,\dots,r\},\ j_0 <\dots <j_{r-1}  \}, $$
defined by setting $f(s)=J$ with $j_r=s$. Since $j_0<\dots <j_{r-1}$, a choice of $j_r$ determines the rest of the components of $J$. It is clear that $f$ is 1:1 and onto.
By applying this correspondence to the formula for $\pa \sigma$ we obtain
\begin{eqnarray}
\pa \sigma = \sum_{s=0}^r(i_0,\dots,\widehat{i_s},\dots i_r) \nonumber
&=& \sum_{J:j_0<\dots<j_{r-1}}(-1)^{j_r}(i_0,\dots,\widehat{i_{j_r}},\dots,i_r) \nonumber \\
&=& \sum_{J:j_0<\dots<j_{r-1}}(-1)^{j_r}(i_{j_0},\dots,\widehat{i_{j_{r-1}}},\dots,i_{j_r}) \nonumber
\end{eqnarray}
where the last equality holds since $(i_0,\dots,\widehat{i_{j_r}},\dots,i_r)=(i_{j_0},\dots,\widehat{i_{j_{r-1}}},\dots,i_{j_r})$. Indeed,
there exists $l$, $0\leq l\leq r-1$ such that $j_{l-1}<j_r<j_l$. In particular, it follows that $j_s=s$ for $0\leq s \leq l-1$, $j_s=s+1$ for $l \leq s\leq r-1$ and $j_r=l$.

Observe that $sign(J)$ equals to $(-1)$ raised the power of the number of transpositions that are needed to move $j_r$ to its position in $J$ which is equal to $l=j_r$. Therefore, 
$(-1)^{j_r}=sign(J)$. The second identity is proven the same way.
\end{proof}
Next, we give a proof of Proposition \ref{verify_criterion}.\\
\begin{conv}
In this subsection we use letters $J,J^1,\dots$ to denote permutations of 
$\{0,\dots,r\}$ letters $K,K^1,\dots$ to denote permutations of $\{0,\dots,r-1\}$.
Components of a permutation $J^l$ we be denoted by $(j^l_0,\dots,j^l_{r})$ ( similarly for $K^l$). 
If $\sigma$ is a simplex then we write $\sigma=(i_0(\sigma),\dots,i_r(\sigma))$ or just $(i_0,\dots,i_r)$ when there is no confusion possible.
\end{conv}
$ $\\
\textit{Proof of Proposition \ref{verify_criterion}. }
Using the formula in Lemma \ref{formula} with $m=r-1$ and noting that $\pa Sd_r(J)=\sigma^b - (i_{j_0},\dots,i_{j_{r-1}})^b$ where $\sigma^b$ is the barycenter of $\sigma$
we obtain
\begin{eqnarray}
\int_\sigma \omega &=& 
(-1)^{\lfloor \frac{r}{2} \rfloor}\sum_{J:j_0<\dots<j_{r-1}}sign(J)
\left[\xi^{r-1}_{i_{j_0},\dots,i_{j_{r-1}}}(\sigma^b) - 
\xi^{r-1}_{i_{j_0},\dots,i_{j_{r-1}}}((i_{j_0},\dots,i_{j_{r-1}})^b) \right]\nonumber \\
%\int_{\pa Sd_{r}(J)}\xi^{r-1}_{i_{j_0},\dots,i_{j_{r-1}}}\nonumber \\ 
&+& 
(-1)^r\sum_{J:j_0<\dots<j_{r-1}}sign(J)\sum_{t=0}^{r-2}(-1)^{\lfloor \frac{t}{2} \rfloor}
\sum_{K:k_0<\dots<k_t}sign(k)\int_{Sd_{t+1}(J'_K)}\xi^t_{i_{j_{k_0}},\dots,i_{j_{k_t}}}\nonumber\\
&=&
(-1)^{\lfloor \frac{r}{2} \rfloor}\left[(\delta\xi^{r-1})_{i_0,\dots,i_r} - 
\sum_{J:j_0<\dots<j_{r-1}}sign(J) \xi^{r-1}_{i_{j_0},\dots,i_{j_{r-1}}}((i_{j_0},\dots,i_{j_{r-1}})^b) \right]\nonumber \\
&+& 
(-1)^r\sum_{J:j_0<\dots<j_{r-1}}sign(J)\sum_{t=0}^{r-2}(-1)^{\lfloor \frac{t}{2} \rfloor}
\sum_{K:k_0<\dots<k_t}sign(k)\int_{Sd_{t+1}(J'_K)}\xi^t_{i_{j_{k_0}},\dots,i_{j_{k_t}}}\nonumber\\
&=&
(-1)^{\lfloor \frac{r}{2} \rfloor}(\delta\xi^{r-1})_{i_0,\dots,i_r}\nonumber\\
&+&
(-1)^r\sum_{J:j_0<\dots<j_{r-1}}sign(J)\sum_{t=0}^{r-1}(-1)^{\lfloor \frac{t}{2} \rfloor}
\sum_{K:k_0<\dots<k_t}sign(k)\int_{Sd_{t+1}(J'_K)}\xi^t_{i_{j_{k_0}},\dots,i_{j_{k_t}}}\nonumber
\end{eqnarray}
where the second equality is obtained by application of Lemma \ref{bdry_rep}.
Set 
$$ \psi(J,\sigma,\omega):= (-1)^r\sum_{t=0}^{r-1}(-1)^{\lfloor \frac{t}{2} \rfloor}
\sum_{K:k_0<\dots<k_t}sign(k)\int_{Sd_{t+1}(J'_K)}\xi^t_{i_{j_{k_0}},\dots,i_{j_{k_t}}}. $$
So we get the following identity
$$ (-1)^{\lfloor \frac{r}{2} \rfloor}(\delta\xi^{r-1})_{i_0,\dots,i_r} = \int_\sigma \omega - 
\sum_{J:j_0<\dots<j_{r-1}}sign(J)\psi(J,\sigma,\omega). $$
Therefore,
$$
\sum_\sigma \eps_\sigma(\delta\xi^{r-1})_{i_0,\dots,i_r} =
(-1)^{\lfloor \frac{r}{2} \rfloor}\sum_\sigma \eps_\sigma \left(\int_\sigma \omega - \sum_{J:j_0<\dots<j_{r-1}}sign(J)\psi(J,\sigma,\omega)\right) 
$$
Since $a$ is a cycle and $\omega$ is an exact form it follows that $\sum_\sigma \eps_\sigma \int_\sigma \omega=\int_a \omega = 0$.
By Proposition \ref{bdry_rep} we have
\begin{equation}\label{pa_a}
0=\pa a = \sum_\sigma \eps_\sigma \sum_{J:j_0<\dots<j_{r-1}}sign(J)(i_{j_0}(\sigma),\dots,i_{j_{r-1}}(\sigma))
\end{equation}
and thus, 
\begin{equation}\label{thus1}
\sum_\sigma \eps_\sigma  \sum_{J:j_0<\dots<j_{r-1}}sign(J)\psi(J,\sigma,\omega)=0. 
\end{equation}
Formula (\ref{thus1}) can be explained as follows. We can define an $(r-1)$ co-chain $g_\psi$ as follows.
For any $(r-1)$ simplex $\tau=(i_0,\dots,i_{r-1})$ define 
$$ g_\psi(\tau) := \psi ((0,\dots,r),\sigma,\omega) $$
where $\sigma=(i_0,\dots,i_{r-1},i_r)$ is any $r$ simplex containing $\tau$ in its boundary.
Co-chain $g_\psi$ is well defined since $\psi$ depends only only on $(i_0(\sigma),\dots,i_{r-1}(\sigma))=\tau $. 
Any pair $(\sigma,J)$ where $\sigma$ is an $r$-simplex and $J$ is a permutation defines
an $(r-1)$-simplex $\tau(\sigma,J):=(i_{j_0}(\sigma),\dots,i_{j_{r-1}}(\sigma)) $.
Note that $g_\psi(\tau(\sigma,J))=\psi(J,\sigma,\omega)$.
Now, from (\ref{pa_a}) it follows that
\begin{eqnarray}\label{pa_aa}
0 &=& g_\psi(\pa a)\\ &=& \nonumber
g_{\psi}\left(\sum_\sigma \eps_\sigma  \sum_{J:j_0<\dots<j_{r-1}}sign(J)(i_{j_0}(\sigma),\dots,i_{j_{r-1}}(\sigma))\right)
\\&=&\nonumber
\sum_\sigma \eps_\sigma  \sum_{J:j_0<\dots<j_{r-1}}sign(J)g_{\psi}(\tau(\sigma,J))
\\&=&\nonumber 
\sum_\sigma \eps_\sigma  \sum_{J:j_0<\dots<j_{r-1}}sign(J)\psi(J,\sigma,\omega).
\nonumber
\end{eqnarray}

$\square$\\

We turn now to the proof of Lemma \ref{formula} which is a consequence of the next formula.
\begin{lem}\label{formula2}
In the above notation we have the following formula
\begin{eqnarray}\label{formula3}
\sum_{J:j_0< \dots <j_t} sign(J)\int_{\pa Sd_{t+1}(J)}\xi^t_{i_{j_0},\dots,i_{j_t}} &=& \\
(-1)^{t+1}\sum_{J:j_0< \dots <j_{t+1}} sign(J)\int_{\pa Sd_{t+2}(J)}\xi^{t+1}_{i_{j_0},\dots,i_{j_{t+1}}} 
&+& \nonumber \\
(-1)^{r-t}\sum_{J:j_0< \dots <j_{r-1}} sign(J)\sum_{K:k_0<\dots<k_t}sign(K)\int_{Sd_{t+1}(J'_K)}\xi^t_{i_{j_{k_0}},\dots,i_{j_{k_t}}}.\nonumber
\end{eqnarray}
\end{lem}
\begin{proof}
Let $J=(j_0,\dots,j_r)$ be fixed with $j_0<\dots<j_t$.  
Set 
$$ \alpha(J,t,s) := \left((i_{j_0},\dots,i_{j_t})^b,\dots,\widehat{(i_{j_0},\dots,i_{j_s})^b},\dots
(i_{j_0},\dots,i_{j_r})^b \right), $$
where $t\leq s \leq r$.
We may represent 
$$ \pa Sd_{t+1}(J) = \sum_{s=t}^{r} (-1)^{s-t} \alpha (J,t,s) $$
and thus we may write the left hand side of (\ref{formula2}) as
\begin{eqnarray}\label{formula4}
\sum_{J:j_0< \dots <j_t} sign(J)\int_{\pa Sd_{t+1}(J)}\xi^t_{i_{j_0},\dots,i_{j_t}} &=&\\
\sum_{J:j_0< \dots <j_t} sign(J)\sum_{s=t}^{r}(-1)^{s-t} \int_{\alpha(J,t,s)}\xi^t_{i_{j_0},\dots,i_{j_t}}.
\nonumber
\end{eqnarray}
Note that when $t<s<r$ there exist exactly two different permutations $J^{(1)}=J=(j_0,\dots,j_r)$ and
$J^{(2)}=(j_0,\dots,j_{s-1},j_{s+1},j_s,j_{s+2},\dots,j_r)$ such that
$\alpha(J^{(1)},t,s)=\alpha(J^{(2)},t,s)$ also observe that $sign(J^{(1)})=-sign(J^{(2)})$.
It follows that 
\begin{eqnarray}
\sum_{J:j_0< \dots <j_t} sign(J)\sum_{s=t}^{r}(-1)^s \int_{\alpha(J,t,s)}\xi^t_{i_{j_0},\dots,i_{j_t}}
&=&\nonumber \\
\sum_{J:j_0< \dots <j_t} sign(J)\left(\int_{\alpha(J,t,t)}\xi^t_{i_{j_0},\dots,i_{j_t}} 
+
(-1)^{r-t}\int_{\alpha(J,t,r)}\xi^t_{i_{j_0},\dots,i_{j_t}}
\right) 
&=&\nonumber\\
\sum_{J:j_0< \dots <j_t} sign(J)\left(\int_{Sd_{t+2}(J)}\xi^t_{i_{j_0},\dots,i_{j_t}} 
+
(-1)^{r-t}\int_{Sd_{t+1}(J')}\xi^t_{i_{j_0},\dots,i_{j_t}}
\right).
\nonumber
\end{eqnarray}
First we show that
\begin{eqnarray}\label{formula_part_1}
\sum_{J:j_0< \dots <j_t} sign(J)\int_{Sd_{t+2}(J)}\xi^t_{i_{j_0},\dots,i_{j_t}} 
&=&\\
(-1)^{t+1}\sum_{J:j_0< \dots <j_{t+1}} sign(J)\int_{\pa Sd_{t+2}(J)}\xi^{t+1}_{i_{j_0},\dots,i_{j_{t+1}}}.
\nonumber
\end{eqnarray}
If $t$ is fixed and  $J=(j_0,\dots,j_r)$ is a permutation denote by $(J_s)$ the permutation $\left( (J_s)_0,\dots,(J_s)_r \right) $ where
$$
(J_s)_l = 
\left\{\begin{array}{rl}
 j_l     &\mbox{ if $0\leq l\leq s-1$} \\
 j_{l+1} &\mbox{ if $s \leq l \leq t$} \\
 j_s     &\mbox{ if $l=t+1$ }\\
 j_l     &\mbox{ if $t+2 \leq l \leq r$} 
\end{array} \right. ,
$$
where $0 \leq s\leq t+1$. Permutation $(J_s)$ is obtained from $J$ by removing $j_s$ and replacing it between $j_{t+1}$ and $j_{t+2}$. Note that if $s=t+1$ then $(J_s)=J$.
Observe that 
\begin{equation}\label{sdjs}
Sd_{t+2}(J)=Sd_{t+2}((J_s)) 
\end{equation}
since 
$\{(J_s)_0,\dots,(J_s)_{t+1} \}=\{j_0,\dots,j_{t+1} \} $.
There exists a 1:1 and onto correspondence 
$$f:\{(J,s): s=0,\dots,t+1, \ \ \ j_0<\dots<j_{t+1} \} \to \{ J : j_0<\dots<j_t \}, $$
given by $(J,s)\mapsto (J_s)$.

First we show that $f$ is $1:1$. Suppose that $(J^1,s^1)$ and $(J^2,s^2)$ are two different elements that are mapped to $\tilde{J}=(\tilde{j}_0,\dots,\tilde{j}_r)$. So, $\tilde{j}_{t+1} = j^1_{s^1}=j^2_{s^2} $. But since $\tilde{j}_0<\dots<\tilde{j}_r$ there exists unique $l$, $0\leq l\leq t$, such that $\tilde{j}_l<\tilde{j}_{t+1}<\tilde{j}_{l+1}$ if $l<t$ and $\tilde{j}_t<\tilde{j}_{t+1}$ 
if $l=t$. That is, there exists a unique place between the elements of the sequence $\tilde{j}_0,\dots,\tilde{j}_t$ where $\tilde{j}_{t+1}$ can be squeezed to form an ordered sequence.
It follows that $(J^1,s^1)=(J^2,s^2)$.

The correspondence $f$ is onto. Let $J=(j_0,\dots,j_r)$, $j_0,\dots,j_t$. If $j_{t+1}>j_t$ then set $s:=t+1$. Otherwise, there exists a unique $s$ such that $j_{s-1}<j_{t+1}<j_s$.
Therefore, $J = f(\tilde{J},s)$ where 
$$
\tilde{J}_l = 
\left\{\begin{array}{rl}
 j_l     &\mbox{ if $0\leq l\leq s-1$} \\
 j_{t+1}     &\mbox{ if $l=s$} \\
 j_{l-1} &\mbox{ if $s+1 \leq l \leq t+1$} \\
 j_l     &\mbox{ if $t+2 \leq l \leq r$} 
\end{array} \right. .
$$
Permutation $\tilde{J}$ is obtained from $J$ by swapping $j_s$ with $j_{t+1}$.
Next we use correspondence $f$ in the left hand side of formula (\ref{formula_part_1})
\begin{eqnarray}
\sum_{J:j_0< \dots <j_t} sign(J)\int_{Sd_{t+2}(J)}\xi^t_{i_{j_0},\dots,i_{j_t}} 
&=&\nonumber\\
\sum_{J:j_0< \dots <j_{t+1}}\sum_{s=0}^{t+1} sign((J_s))\int_{Sd_{t+2}((J_s))}\xi^t_{i_{j_0},\dots\hat{i}_{j_s}\dots,i_{j_{t+1}}} 
&=&\nonumber\\
\sum_{J:j_0< \dots <j_{t+1}}\sum_{s=0}^{t+1} sign(J) (-1)^{t-s+1}\int_{Sd_{t+2}(J)}\xi^t_{i_{j_0},\dots\hat{i}_{j_s}\dots,i_{j_{t+1}}} 
&=&\nonumber\\
(-1)^{t+1}\sum_{J:j_0< \dots <j_{t+1}} sign(J) \int_{Sd_{t+2}(J)}(\delta\xi^t)_{i_{j_0},\dots,i_{j_{t+1}}} 
&=&\nonumber\\
(-1)^{t+1}\sum_{J:j_0< \dots <j_{t+1}} sign(J) \int_{\pa Sd_{t+2}(J)}\xi^{t+1}_{i_{j_0},\dots,i_{j_{t+1}}}. 
\nonumber
\end{eqnarray}
where in the second equality we used equation (\ref{sdjs}) and $sign((j_s))=(-1)^{t-s+1}sign(J)$ since it takes $t-s+1$ transpositions to transform $J$ into $(J_s)$. 
In the third equality we used the definition of $\delta \xi^t$ and in the last equality we used Stokes formula after using a relation defining $\xi^{t+1}$ (see Definition \ref{def_xi}).

Next we prove 
\begin{eqnarray}\label{formula_part_2}
\sum_{J:j_0< \dots <j_t} sign(J)\int_{Sd_{t+1}(J')}\xi^t_{i_{j_0},\dots,i_{j_t}}
&=&\\
\sum_{J:j_0< \dots <j_{r-1}} sign(J)\sum_{K:k_0<\dots<k_t}sign(K)\int_{Sd_{t+1}(J'_K)}\xi^t_{i_{j_{k_0}},\dots,i_{j_{k_t}}}.
\nonumber
\end{eqnarray}
There exists a 1:1 and onto correspondence
$$h: \{(J,K): j_0<\dots<j_{r-1},\ \ k_0<\dots<k_t  \}\to\{J: j_0<\dots<j_{t}\}, $$
defined by $(J,K)\mapsto (j_{k_0},\dots,j_{k_{r-1}},j_r )=(J'_K,j_r)$.\\
{\it $h$ is 1:1.} Suppose that $(J^1,K^1)$ and $(J^2,K^2)$ are mapped to 
$(j_0,\dots,j_r)$. Then we must have $j_r=j^1_r=j^2_r$. It follows that 
$\{j^1_0,\dots,j^1_{r-1} \}=\{j^2_0,\dots,j^2_{r-1} \}$. But since 
$j^l_0<\dots<j^l_{r-1}$ for $l=1,2$ we must have $J^1=J^2$.
Since $K^1$ and $K^2$ are permutations, $(J^1)'_{K^1}=(J^2)'_{K^2}$ and $(J^l)'$, $l=1,2$ have distinct elements it follows that $K^1=K^2$.\\
{\it $h$ is onto.} Let $J$, $j_0<\dots<j_t$ be a permutation.
Set $J^1$ to be a permutation with $j^1_0<\dots<j^1_{r-1}$, $j^1_r:=j_r$ and 
$\{j_0,\dots,j_{r-1} \}=\{j^1_0,\dots,j^1_{r-1} \}$.
There exists a permutation $K$ of $\{0,\dots,r-1 \}$ such that $j_i=j^1_{k_i}$.
Now, $(J^1,K)\mapsto J$.

We apply correspondence $h$ to left hand side of (\ref{formula_part_2}) obtaining
\begin{eqnarray}
\sum_{J:j_0< \dots <j_t} sign(J)\int_{Sd_{t+1}(J')}\xi^t_{i_{j_0},\dots,i_{j_t}}
&=&\nonumber\\
\sum_{J:j_0< \dots <j_{r-1}}\sum_{K:k_0< \dots <k_t} sign(J'_K,j_r)\int_{Sd_{t+1}(J'_K)}\xi^t_{i_{j_{k_0}},\dots,i_{j_{k_t}}}.
\nonumber
\end{eqnarray}
Finally, note that $sign(K)sign(J)=sign(J'_K,j_r)$ and the lemma is proven.
\end{proof}

We are ready now to prove Lemma \ref{formula}.\\
{\it Proof of Lemma \ref{formula}.} 
The proof is by induction on $m$. For $m=0$ we have
\begin{eqnarray}
\int_{\sigma}\omega = \sum_J sign(J)\int_{Sd_1(J)} d\xi^{0}_{i_{j_0}}= 
\sum_J sign(J)\int_{\pa Sd_1(J)} \xi^{0}_{i_{j_0}}.\nonumber
\end{eqnarray}
Suppose that $m>0$. Induction hypothesis for $m-1$ followed by application of Lemma \ref{formula2} yields 
\begin{eqnarray}
\int_\sigma \omega &=& (-1)^{\lfloor \frac{m}{2} \rfloor}\sum_{J:j_0<\dots<j_{m-1}}sign(J)
\int_{\pa Sd_{m}(J)}\xi^{m-1}_{i_{j_0},\dots,i_{j_{m-1}}}\nonumber \\ 
&+& %\nonumber \\
(-1)^r\sum_{J:j_0<\dots<j_{r-1}}sign(J)\sum_{t=0}^{m-2}(-1)^{\lfloor \frac{t}{2} \rfloor}
\sum_{K:k_0<\dots<k_t}sign(k)\int_{Sd_{t+1}(J'_K)}\xi^t_{i_{j_{k_0}},\dots,i_{j_{k_t}}}\nonumber\\
&=&
(-1)^{\lfloor \frac{m}{2} \rfloor}
\left[ 
(-1)^{m}\sum_{J:j_0< \dots <j_{m}} sign(J)\int_{\pa Sd_{m+1}(J)}\xi^{m}_{i_{j_0},\dots,i_{j_{m}}} 
\right. \nonumber \\ &+&
\left.
(-1)^{r-(m-1)}\sum_{J:j_0< \dots <j_{r-1}} sign(J)\sum_{K:k_0<\dots<k_{m-1}}sign(K)\int_{Sd_{m}(J'_K)}\xi^t_{i_{j_{k_0}},\dots,i_{j_{k_{m-1}}}}\nonumber
\right] \nonumber \\
&+& %\nonumber \\
\sum_{J:j_0<\dots<j_{r-1}}sign(J)\sum_{t=0}^{m-2}(-1)^{\lfloor \frac{t}{2} \rfloor}
\sum_{K:k_0<\dots<k_t}sign(k)\int_{Sd_{t+1}(J'_K)}\xi^t_{i_{j_{k_0}},\dots,i_{j_{k_t}}}\nonumber\\
&=&
(-1)^{\lfloor \frac{m+1}{2} \rfloor}
\sum_{J:j_0< \dots <j_{m}} sign(J)\int_{\pa Sd_{m+1}(J)}\xi^{m}_{i_{j_0},\dots,i_{j_{m}}} 
\nonumber \\ 
&+&
(-1)^r\sum_{J:j_0<\dots<j_{r-1}}sign(J)\sum_{t=0}^{m-1}(-1)^{\lfloor \frac{t}{2} \rfloor}
\sum_{K:k_0<\dots<k_t}sign(k)\int_{Sd_{t+1}(J'_K)}\xi^t_{i_{j_{k_0}},\dots,i_{j_{k_t}}},\nonumber
\end{eqnarray}
where the last equality is since the parity of ${\lfloor \frac{m}{2} \rfloor}+m =$ parity of 
${\lfloor \frac{m+1}{2} \rfloor}$ and the parity of $(m-1)+{\lfloor \frac{m}{2} \rfloor}=$
parity of ${\lfloor \frac{m-1}{2} \rfloor}$
and the Lemma is proven. $\square$
%\end{comment}
%
%
%
%
%
%
%
%
\section{$L^{\overline p}$-cohomology}\label{Lp_cohomology}
In this section we study an $L^{\overline p}$ complex, introduced in [GoTr].
We show that that Poincar\'e inequality for differential forms on a manifold $X$, implies 
an embedding of $L^{\overline p}$-cohomology of $X$ into the standard De Rham cohomology.
The main tool we use here is a method developed by B. Yousin [Y] to regularize the forms.

\begin{df}
Let $\overline p=(p_0,\dots,p_n)$, $1\leq p_i\leq\infty$, be an $n$-tuple of numbers. 
Set 
$$ \Omega^k_{\overline p}(X):=\{ \omega \in L^{p_k}(\Omega^k): d\omega\in L^{p_{k+1}}(\Omega^{k+1}) \}. $$ 
The complex $(\Omega^{\bullet}_{\overline p}(X),\overline d)$ where $\overline d$ is the weak exterior derivative,
is called {\bf $L^{\overline p}$ De Rham complex}.

The cohomology ring of $(\Omega^{\bullet}_{\overline p}(X),\overline d)$ is denoted by $H_{\overline p}^\bullet(X)$ and 
called  {\bf $L^{\overline p}$ cohomology} of $X$.

We say that the complex $\Omega^{\bullet}_{\overline p}(X)$ satisfies generalized Poincar\'e inequality if
there exist positive constants $C_k$ such that for each $k$-form $\omega$ such that $d\omega\in L^{p_{k+1}}$ there exists a closed 
$k$-form $\omega_0$
such that
$$ \| \omega -\omega_0\|_{p_k}\leq C_k \|d\omega \|_{p_{k+1}} .$$
\end{df}

\begin{thm}\label{Y_1}([Y] Theorem 2.7.1)
Let $X$ be a smooth Riemaniann manifold and $\omega\in L^p(X)$ a $k+1$-form with $\overline d\omega\in \C^\infty(X)$.
Then for any $\eps>0 $ there exists a $k$-form $\psi_\eps$ such that
\begin{enumerate}
\item $\|\psi_\eps\|_{p}<\eps$, 
\item $\|\overline d\psi_\eps\|_{p}<\eps$,
\item $\omega+ \overline d\psi_\eps $ is smooth.
\end{enumerate}
\end{thm}

\begin{cor}\label {cor_Y}
Let $\Omega_{\overline p}^{\bullet}(X)$ be an $L^{\overline p}$ complex that satisfies Poincar\'e inequality.
Suppose that $\omega\in\Omega_{\overline p}^{k+1}$ and $d\omega$ are smooth on $X$. 
Then, there exists a $k$-form $\psi$ such that 
\begin{enumerate}
\item $\|\psi\|_{{p_{k}}}<\infty$, 
\item $\|\overline d\psi\|_{{p_{k+1}}}<\infty$,
\item $\omega+ \overline d\psi $ is smooth.
\end{enumerate}
\end{cor}
\begin{proof}

By Theorem \ref{Y_1} there exists $\psi_0$ such that 
$\|\psi_0 \|_{p_{k+1}}<\infty$, $\|\dbar \psi_0 \|_{p_{k+1}}<\infty$ and $\omega+\dbar \psi_0$ is smooth. By generalized Poincar\'e inequality there exists a closed $k$-form $\psi_1$ such that 
$$ \| \psi_0 - \psi_1 \|_{{p_{k}}}\lesssim \| \dbar \psi_0 \|_{{p_{k+1}}}.$$
Set $$ \psi:=\psi_0 -\psi_1. $$ 
Now,
$\dbar \psi = \dbar \psi_0 $ so $\omega+\dbar \psi$ is smooth.
We also have
$$\| \psi \| _{{p_k}} =\| \psi_0 - \psi_1 \|_{{p_{k}}}\lesssim \| \dbar\psi_0 \|_{{p_{k+1}}}<\infty . $$
\end{proof}

\begin{thm}\label{inc_thm}
There is a natural inclusion $i:H_{\overline p}^\bullet(X)\to H_{DR}^\bullet(X)$ where $H_{DR}^\bullet(X)$ is the De Rham cohomology group. 
\end{thm}
\begin{proof}
First we show existence of a map $i:H_{\overline p}^\bullet(X)\to H^\bullet(X)$.
Let $\omega$ be a closed form in $\Omega_{\overline p}^{k+1}(X)$ then by Corollary \ref{cor_Y} there exits
$\psi$ such that $\alpha:=\omega+\dbar\psi$ is smooth, $\|\psi\|_{p_k}<\infty$ and $\|d\psi\|_{p_{k+1}}<\infty$. So $\alpha$ is smooth and defines the same cohomology class as $\omega$. Set $i[\omega]$ to be the cohomology class of $\alpha$ in $H^{k+1}(X)$.
We claim that the map $i$ is well defined. First we show that $i$ is independent of choice of the form $\psi$. Suppose that $\psi'$ is another form such that $\alpha':=\omega+\dbar\psi'$ is smooth, $\|\psi'\|_{p_k}<\infty$ and $\|d\psi'\|_{p_{k+1}}<\infty$. 
We have 
$$\alpha-\alpha'=\dbar(\psi-\psi') $$
Since $\dbar(\psi-\psi')$ is smooth, there exists a form $\xi\in\Omega^{k-1}_{\overline p}$ such that $\psi-\psi'+d\xi$ is smooth and therefore, $\alpha-\alpha'=d(\psi-\psi'+d\xi)$.

Next, we show that $i$ is independent of the representative $\omega$. Suppose that $\omega''$ is another form from the cohomology class of $\omega$ in $H_{\overline p}^{k+1}(X)$. 
Since we proved independence from $\psi$, we may assume that $\omega$ and $\omega''$ are smooth.
Then there exists a form  $\gamma\in\Omega^{k}_{\overline p}$ such that $\omega-\omega''=d\gamma$.
Since $d\gamma$ is smooth we may find a form $\beta$ such that $\gamma+\dbar \beta$ is smooth. 
The claim follows.

Finally we show injectivity of $i$. Suppose that $\omega\in\Omega^{k+1}_{\overline p}$ is a smooth closed form 
such that
$\omega=d\gamma$ where $\gamma$ is smooth. By generalized Poincar\'e inequality there exists a closed form $\gamma_0$ such that 
$$\| \gamma-\gamma_0\|_{p_{k}}\lesssim \|\omega\|_{p_{k+1}}.$$
It follows that $\gamma-\gamma_0\in\Omega^{k}_{\overline p} $ and $ d(\gamma-\gamma_0)=\omega $.

\end{proof}
%\end{comment}

Note that the inclusion in Theorem \ref{inc_thm} is not surjective in general.
\begin{ex}
Let $X$ be an open punctured disk in $\R^2$ and $\overline p := (p,p,p)$. 
It is well known that $H^1_{DR}(X)$ is one dimensional and spanned by the angle form 
$\omega =\frac{xdy -y dx}{x^2+y^2}$ where $x,y$ are the standard coordinates in $\R^2$.
A straight forward computation shows that for $p\geq 2$ the form $\omega$ is not $L^p$ bounded
and therefore $H^1_{\overline p}(X)=0$. Indeed, suppose that $\alpha$ is a smooth $1$-form representing a non trivial element of $H^1_{\overline p}(X)$. 
Since the disk $X$ can be covered by three sectors each of which convex, 
it follows that $X$ satisfies Poincar\'e inequality. Therefore,
there exists a non zero number $a$ such that $\int_{c_r} \alpha = a$ for any circle $c_r$ of radius $r$ around the origin. It follows that 
$$|a| = | \int_{c_r} \alpha | \leq \int_{c_r} |\alpha| \leq (\int_{c_r}|\alpha|^p)^{1/p}(2\pi r)^{1/p'}$$
and hence
$$ \int_{c_r}|\alpha|^p \geq \frac{|a|^p}{(2\pi r)^{p-1}}. $$
Therefore, writing the norm of $\alpha$ in polar coordinate $(r,\theta)$ we obtain
$$\| \alpha \|^p_{L^p} = \int_0^1 dr \int_{c_r} |\alpha|^p d\theta \geq \frac{|a|^p}{{(2\pi)}^{p-1}}\int_0^1 \frac{dr}{r^{p-1}} = \infty ,$$
in contradiction with the fact that $\|\alpha \|_{L^p}<\infty$.

\end{ex}
%\medskip {2mm}
As a consequence of Theorem \ref{inc_thm} we obtain the following
\begin{cor}
If $X$ is a compact manifold then $i:H_{\overline p}^\bullet(X)\to H_{DR}^\bullet(X)$
is an isomorphism
\end{cor}
\begin{proof}
By Theorem \ref{inc_thm} the map $i$ is injective so we only need to show that it is surjective.
But a smooth form on compact manifold is bounded and therefore $L^p$ bounded.
\end{proof}


\begin{thebibliography}{VVV}

%\bibitem[H] {key-1}A. Hatcher, Algebraic Topology, Cambridge university press.

\bibitem[BoMi] {key-43}
L.P. Bos and P.D. Milman, Sobolev-Gagiliardo-Nirenberg and Markov type inequalities on subanalytic domains, Geometric And Functional Analysis, Vol. 5, No. 6 (1995).


\bibitem[BT]{key-2}Bott, Tu, Differential forms in algebraic topology. Graduate Texts
in Mathematics, Vol. 82, Springer-Verlag, New York, 1982, xiv + 331 pp.


\bibitem[D] {key-234} do Carmo, M. P. Riemannian Geometry, Boston, Birkh\"auser, 1992.

\bibitem[GoTr]{key-22} V. Gol'dshtein and M. Troyanov, Sobolev inequalities for differential forms and $L_{q,p}$-cohomology, Journal of Geometric Analysis, Vol. 16, No.4 (2006).

\bibitem[IwLu] {key-2} T. Iwaniec and A. Lutoborski, Integral Estimates for Null Lagrangians,
Arch. Rational Mech. Anal. 125 (1993) 25-79. Springer-Verlag 1993.


\bibitem [S] {key-8} L. Shartser,  Explicit proof of Poincar\'e inequality for differential forms on manifolds, C. R. Math. Rep. Acad. Sci. Canada, 2010 (to appear)

\bibitem [Y] {key-8} B. Youssin,  $L^p$-cohomology of cones and horns,  J. Differential Geom. Volume 39, Number 3 (1994), 559-603. 

\end{thebibliography}
\end{document}